\documentclass[12pt]{amsart}
\usepackage[utf8]{inputenc}
\usepackage{amsmath,amssymb}
\usepackage{MnSymbol}

\newtheorem{theorem}{Theorem}
\newtheorem{conjecture}{Conjecture}
\theoremstyle{definition}
\newtheorem{defi}{Definition}

\begin{document}
\begin{abstract}
Wang and Sun proved a certain summatory formula involving derangements and primitive roots of the unit. We study
such a formula but for the particular case of the set of affine derangements in $\overrightarrow{GL}(\mathbb{Z}/2k\mathbb{Z})$
and its subset of involutive affine derangements in particular; in this last case its value is relatively simple and it is related to even unitary divisors of $k$.
\end{abstract}

\title{Wang-Sun Formula in $\overrightarrow{GL}(\mathbb{Z}/2k\mathbb{Z})$}
\author{
Octavio A. Agustín-Aquino
}
\address{Universidad Tecnológica de la Mixteca. Carretera a Acatlima Km 2.5 s/n. Huajuapan de León, Oaxaca, México, C. P. 69000.}
\email{octavioalberto@mixteco.utm.mx}
\date{29 July 2022}
\maketitle

\section{Introduction}
In \cite{WS22}, for $n>1$ an odd integer and $\zeta$ a $n$-th primitive root of unity, Wang and Sun proved
that
\[
\sum_{\pi\in D(n-1)}\mathrm{sign}(\pi)\prod_{j=1}^{n-1}\frac{1+\zeta^{j-\pi(j)}}{1-\zeta^{j-\pi(j)}}
= (-1)^{\frac{n-1}{2}}\frac{((n-2)!!)^{2}}{n}
\]
where $D(n-1)$ is the set of all derangements within $S_{n-1}$.

For the mathematical theory of counterpoint it is not the whole of $S_{n-1}$ as important as it
is the general affine linear group
\[
 \overrightarrow{GL}(\mathbb{Z}/2k\mathbb{Z}):=\{e^{u}.v\}_{u\in\mathbb{Z}/2k\mathbb{Z},v\in\mathbb{Z}/2k\mathbb{Z}^{\times}}.
\]
where an element $e^{u}.v\in\overrightarrow{GL}(\mathbb{Z}/2k\mathbb{Z})$ maps $x\in \mathbb{Z}/2k\mathbb{Z}$ to
\[
e^{u}.v(x) = vx+u.
\]

We are interested not only in plain derangements but in particular in those that are involutive, i.e., they are their
own inverses. Affine involutive derangements are called \emph{quasipolarities}, and its set within
$\overrightarrow{GL}(\mathbb{Z}/2k\mathbb{Z})$ will be denoted with $Q_{k}$. They are important because they
relate consonances ($K$) and dissonances ($D$) in $2k$-tone equal temperaments; more precisely, if $\mathbb{Z}/2k\mathbb{Z} = K\sqcup D$ and $q$ is a quasipolarity, then it is the unique affine morphism such that $q(K)=D$ (see \cite{AJM15} for details).

The structure of the article is as follows: first we prove the sum for quasipolarities in Section \ref{S:Q}. Then
we calculate using code in Maxima some values of the sum for derangements in $\overrightarrow{GL}(\mathbb{Z}/2k\mathbb{Z})$
in Section \ref{S:D} and finally we make some remarks in Section \ref{S:C}.

\section{The case of quasipolarities}\label{S:Q}

Note first that all quasipolarities have the same sign as permutations, so we can disregard
it in the original Wang-Sun formula.

\begin{theorem}\label{T:WangSunSumQEven} If $k$ is odd and $\zeta$ a $2k$-th primitive root of unity, then
\[
\sum_{\pi\in Q_{k}}\prod_{j=0}^{n-1}\frac{1+\zeta^{j-\pi(j)}}{1-\zeta^{j-\pi(j)}} = 0.
\]
\end{theorem}
\begin{proof}
Whenever $j-\pi(j) = k$ then
\[
 1+\zeta^{j-\pi(j)} = 1+\zeta^{k} = 1-1 = 0,
\]
thus it suffices to prove that this happens for any quasipolarity $\pi = e^{u}.v\in Q_{k}$ for some $0\leq j\leq 2k-1$. From
\cite[Theorem 3.1]{oA12}, we know that
\begin{equation}\label{Eq:Carac}
 u = \sigma(v) q + \frac{2k}{\tau(v)}
\end{equation}
where $\sigma(v) = \gcd(v-1,2k)$, $\tau(v) = \gcd(v+1,2k)$ and $q$ is any integer. Now, we want to show that
there is a $j$ such that
\[
 j-\pi(j) = j-e^{u}.v(j) = (1-v)j-u \equiv k \pmod{2k},
\]
which is possible if, and only if, $\gcd(1-v,2k)\backslash (k+u)$. In other words, if and only if
$\sigma(v)\backslash (k+u)$. Using \eqref{Eq:Carac}, this can be rewritten as
\[
 \sigma(v)\backslash \left(k+\sigma(v) q+\frac{2k}{\tau(v)}\right)
\]
for some integer $q$. Thus, if we can show that the congruence
\[
 \sigma(v) x \equiv k-\frac{2k}{\tau(v)} \pmod{2k}
\]
has a solution for $x$, then we will be done. Indeed, $\sigma(v) = 2I_{1}$ and $\tau(v)=2I_{2}$
for some $I_{1},I_{2}$ odd and coprime divisors of $k$. The linear Diophantine equation
\[
 2x+2\left(\frac{k}{I_{1}}\right)y = \left(\frac{k}{I_{1}}-\frac{k}{I_{1}I_{2}}\right)
\]
has a solution since the GCD of the the coefficients on the left is $2$ and the number on the
right is even. Hence, if $(x,y)$ is a solution, then
\[
 2I_{1}x+\frac{k}{I_{2}}-k = \sigma(v)x+\frac{2k}{\tau(v)}-k = 2ky
\]
which shows that $x$ is the required solution.
\end{proof}

\begin{table}[ht]
\caption{Results of the evaluation of the sum $S$ of Theorem \ref{T:WangSunSumQEven} for some values of $k$.}
\label{T:Involutions}
\begin{tabular}{|l|llllllllll|}
\hline
$k$ & $3$ & $4$ & $5$ & $6$ & $7$ & $8$ & $9$ & $10$ & $11$ & $12$\\
\hline
$S$ & $0$ & $4$ & $0$ & $8$ & $0$ & $8$ & $0$ & $12$ & $0$ & $16$ \\
\hline
\end{tabular}
\end{table}

While this result was relatively easy to conjecture considering a few values of the sum
for some $k$ (as they can be seen in Table \ref{T:Involutions}), it is less easy to make a
guess for the sum when $k$ is even, but the following seems reasonable.

\begin{conjecture}\label{Conj:Even} If $\zeta$ is a $2k$-th primitive root of unity, then
\[
\sum_{\pi\in Q_{k}}\prod_{j=0}^{n-1}\frac{1+\zeta^{j-\pi(j)}}{1-\zeta^{j-\pi(j)}} =
\sum_{\pi\in Q_{k}}\prod_{j=0}^{n-1}[j-\pi(j)\neq k].
\]
\end{conjecture}

The search of the sequence $4,8,8,12$ in the Online Encyclopedia of Integer Sequences (OEIS) \cite{oeis} yields
as first result A054785, which is the difference of the sum of divisors of $2n$ and $n$. The values match up to $k=16$, but they differ at $k=18$, where the former is $20$ and the later is $26$.
Nonetheless, the results from \cite{oA14} lead us
in the right direction. We need some definitions first.

\begin{defi}
A divisor $d$ of $n$ is said to be \emph{unitary} if $d\perp n/d$; we denote that
$d$ is a unitary divisor of $n$ with $d\,\backslash\!\backslash n$. We have the sum
of unitary divisors function
\[
s_{1}^{*}(n) = \sum_{d\,\backslash\!\backslash n} d.
\]
\end{defi}

In \cite{oA14} it is proved
that
\[
|Q_{k}| = s_{1}^{*}(k).
\]

Now we can make Conjecture \ref{Conj:Even} more precise.
\begin{theorem}\label{T:WangSunSumQGen} If $\zeta$ is a $2k$-th primitive root of unity, then
\[
\sum_{\pi\in Q_{k}}\prod_{j=0}^{n-1}\frac{1+\zeta^{j-\pi(j)}}{1-\zeta^{j-\pi(j)}} =
[2\backslash k](|Q_{2k}|-|Q_{k}|) = [2\backslash k](s_{1}^{*}(2k)-s_{1}^{*}(k)).
\]
\end{theorem}
\begin{proof}
We have already proved the case when $k$ is odd, so suppose $k$ is even. If $\pi = e^{u}.v\in Q_{k}$
and $j-\pi(j) \nequiv k \pmod{2k}$ so the summand associated to $\pi$ is not $0$, then we require the congruence
\begin{equation}\label{E:NoCong}
 (v-1)j \equiv k-u\pmod{2k}
\end{equation}
to have no solutions for $j$; this happens if and only if $\sigma(v)\medslash\!\!\!\!\!\backslash (k-u)$.
Again, because of \cite[Theorem 3.1]{oA12}, we have $\sigma(v) = \frac{4k}{\tau(v)}$ and
$u = \sigma(v)q+\frac{2k}{\tau(v)}$. Hence \eqref{E:NoCong} has no solutions if and only
if
\begin{equation}\label{E:DoesNotDivide}
\sigma(v)\medslash\!\!\!\!\!\backslash \left(\frac{\sigma(v)\tau(v)}{4}+\sigma(v)q+\frac{\sigma(v)}{2}\right).
\end{equation}

We know $\sigma(v) = 2I_{1}$ and $\tau(v) = 2I_{2}$ where $I_{1}\perp I_{2}$. Moreover
\[
 2k = \frac{\sigma(v)\tau(v)}{2} = \frac{4I_{1}I_{2}}{2} =2I_{1}I_{2}
\]
thus $I_{1}$ and $I_{2}$ are unitary divisors of $k$. By \cite[Proposition 2.3]{oA14}, we know
that both $I_{1}$ and $I_{2}$ represent involutions of $\mathbb{Z}/2k\mathbb{Z}$. If $I_{1}$
is odd, then $I_{2}>0$ is even and the associated involution defines $\frac{2k}{\sigma(v)} = I_{2}$
quasipolarities. In particular, $2I_{1}=\sigma(v)$ divides $k$. Furthermore, $\tau(v)$ is divisible by $4$ and
$\sigma(v)$ does not divide $\sigma(v)/2$, thus \eqref{E:DoesNotDivide} holds. Conversely, if \eqref{E:DoesNotDivide}
is true then $I_{2}$ has to be even for otherwise
\[
 \frac{1}{2}\sigma(v)(2s+1)+\tfrac{1}{2}\sigma(v) = \sigma(v)(s+1)
\]
which contradicts \eqref{E:DoesNotDivide}.

For every $\pi=e^{u}.v$ with its linear part $v$ represented by an even unitary divisor of $k$ and every $0\leq j\leq 2k-1$, we claim that there is an $\ell$
such that $(j-\pi(j))-(\ell-\pi(\ell)) \equiv k \pmod{2k}$, for this is equivalent to
\[
 (v-1)\ell \equiv k-(v-1)j \pmod{2k}.
\]

Since $\sigma(v)\backslash k$ and $\sigma(v)\backslash (1-v)$, it implies that there exist a solution
$\ell$. Therefore, for a summand indexed by $\pi$ such that $j-\pi(j)\neq k$ for
every $0\leq j \leq 2k-1$ and has a factor $1+\zeta^{j-\pi(j)}$ in the denominator, it also has a factor
\[
 1-\zeta^{\ell-\pi(\ell)} = 1-\zeta^{j-\pi(j)+k} = 1-\zeta^{j-\pi(j)}\zeta^{k} = 1+\zeta^{j-\pi(j)}
\]
in the numerator, so all the factors cancel out and the summand equals $1$.

Thus we have to sum the number of even unitary divisors of $k$. This is
accomplished by $s_{1}^{*}(2k) - s_{1}^{*}(k)$ because
\begin{align*}
s_{1}^{*}(2k) - s_{1}^{*}(k) &=
\sum_{d\,\backslash\!\backslash 2k,d\text{ even}}d+
\sum_{d\,\backslash\!\backslash 2k,d\text{ odd}}d-
\sum_{d\,\backslash\!\backslash k,d\text{ even}}d
-\sum_{d\,\backslash\!\backslash k,d\text{ odd}}d\\
&=
2\sum_{d\,\backslash\!\backslash k,d\text{ even}}d+
\sum_{d\,\backslash\!\backslash k,d\text{ odd}}d-
\sum_{d\,\backslash\!\backslash k,d\text{ even}}d
-\sum_{d\,\backslash\!\backslash k,d\text{ odd}}d\\
&=\sum_{d\,\backslash\!\backslash k,d\text{ even}}d.\qedhere
\end{align*}
\end{proof}

\section{The case of the derangements of the general affine group}\label{S:D}
Let us denote the derangements within $\overrightarrow{GL}(\mathbb{Z}/2k\mathbb{Z})$ with
$\Delta_{k}$. The following code in Maxima calculates
\[
S=\sum_{\pi\in \Delta_{k}}\mathrm{sign}(\pi)\prod_{j=0}^{n-1}\frac{1+\zeta^{j-\pi(j)}}{1-\zeta^{j-\pi(j)}}
\]
for the particular case of $k=4$.

\begin{verbatim}
load("combinatorics");
S:0;
kk:4;
zeta: exp(%pi*%i/kk);
ZZ:[];
L:[];
for k1:0 thru (2*kk-1) do ZZ:append(ZZ,[k1]);
for u:0 thru (2*kk-1) do
 (for v:1 thru (2*kk-1) step 2 do
  (if (gcd(v^2,2*kk)=1) then
   (
    M:mod(v*ZZ+u-ZZ,2*kk),
    if(not(product(M[l],l,1,2*kk)=0)) then
     L:append(L,[[u,v]])
   )
  )
 )$
for k1:1 thru length(L) do
 (P: (-1)^perm_parity(mod(L[k1][2]*ZZ+L[k1][1],2*kk)+1),
  for k2:0 thru (2*kk-1) do
   P: trigrat(P*(1+zeta^(k2-mod(L[k1][2]*k2+L[k1][1],2*kk)))
   /(1-zeta^(k2-mod(L[k1][2]*k2+L[k1][1],2*kk)))),
  S:ratsimp(S+P)
 )$
\end{verbatim}

As far as my computer's memory and simplification capacity of Maxima allow, we calculated $S$ for
$k=3,\ldots,9$ \emph{mutatis mutandis}. The results are contained in Table \ref{T:SumWangSunGL}.
Except for the sign and the prime factors of $k$ in the denominator, no other pattern is evident and
searches in the OEIS do not point yet in a meaningful direction.

\begin{table}[ht]
\caption{Results of the execution of the code.}
\label{T:SumWangSunGL}
\begin{tabular}{|l|l|}
\hline
$k$ & $S$\\
\hline
$6$ & $1456/27$\\
$8$ & $-2300$\\
$10$ & $762256/5$\\
$12$ & $-10643506432/729$\\
$14$ & $13444304416/7$\\
$16$ & $-332995177452$\\
$18$ & $1450048309488389824/19683$\\
\hline
\end{tabular}
\end{table}

\section{Some final remarks}\label{S:C}

An interesting outcome of Theorem \ref{T:WangSunSumQEven}, from the musicological viewpoint, is that for $2k$-tone
equal temperaments with $k$ even, there is always a quasipolarity such that a consonance and a dissonance are
separated by a \emph{tritone} (which always corresponds to $k$).

When $k$ is odd, the proof of Theorem \ref{T:WangSunSumQGen} tells us that Wang-Sun sum counts quasipolarities
represented by even unitary divisors along the direction of \cite{oA14}, so it suggest that we should explore more deeply the relationship between the
two concepts and its musicological implications.

For the Wang-Sun sum over all affine derangements we could not prove or conjecture a general formula, and thus we stress the non-triviality of studying it for derangements within subgroups.

\section*{Acknowledgments}

I deeply thank José Hernández Santiago at Universidad Autónoma de Guerrero for his help
completing the proof of Theorem \ref{T:WangSunSumQEven}.
\bibliographystyle{abbrv}
\bibliography{sun_sums}

\begin{thebibliography}{1}

\bibitem{oA12}
O.~A. Agust\'{\i}n-Aquino.
\newblock Antichains and counterpoint dichotomies.
\newblock {\em Contributions to Discrete Mathematics}, 7(2):97--104, 2012.

\bibitem{oA14}
O.~A. Agust\'{\i}n-Aquino.
\newblock Prime injections and quasipolarities.
\newblock {\em Le Matematiche}, 69(1):159--168, 2014.

\bibitem{AJM15}
O.~A. Agust\'{\i}n-Aquino, J.~Junod, and G.~Mazzola.
\newblock {\em Computational Counterpoint Worlds}.
\newblock Springer, 2015.

\bibitem{oeis}
N.~J.~A. Sloane.
\newblock The {O}n-{L}ine {E}ncyclopedia of {I}nteger {S}equences.
\newblock Published electronically at http://oeis.org, 2022.

\bibitem{WS22}
H.~Wang and Z.-W. Sun.
\newblock Proof of a conjecture involving derangements and roots of unity,
  2022.
\newblock arXiv:2206.02589.

\end{thebibliography}

\end{document}